\documentclass[12pt]{amsart}
\usepackage{amsthm,amssymb,mathrsfs,setspace,amsmath, bbm}
\usepackage{a4wide}
\usepackage{hyphenat}
\usepackage{import}
\usepackage{subfiles}
\usepackage[usenames,dvipsnames]{color}
\usepackage{graphicx}
\usepackage{subcaption}
\usepackage[export]{adjustbox}
\usepackage{xcolor}
\usepackage{hyperref}
\hypersetup{
    colorlinks=true,
    linkcolor=blue,
    filecolor=magenta,      
    urlcolor=cyan,
    citecolor=violet
    }

\newtheorem{theorem}{Theorem}[section]
\newtheorem{lemma}{Lemma}[section]

\newtheorem{definition}{Definition}[section]

\title[Negative Lyapunov Exponent of Circle Maps]{Negative Lyapunov Exponent of Circle Maps Forced by Expanding Circle Endomorphisms}
\author{Kirthana Rajasekar}
\address{Department of Mathematics, KTH Royal Insititute of Technology, 100 44 Stockholm, Sweden.}
\email{rajasek@kth.se}
\begin{document}

\begin{abstract} 
We study maps on the torus $\mathbb{T}^2$ that are of the form $F(x,y) = (bx, f_x(y))$, where $b\geq 2$ is an integer. We establish an open class of $C^1$-maps, with $f_x(y)$ that are typically non-monotonic in $x$, for which the Lyapunov exponents on the fibre are negative almost everywhere. For each fixed $f_x(y)$ and a base map $bx$ that is sufficiently expanding, we establish a uniform upper bound for the Lyapunov exponents; moreover, the uniform bound depends on selective characteristics of $f$. This implies that orbits on the same fibre exhibit local synchronisation. 
\end{abstract}
\maketitle
\section{Introduction}
We study skew-product maps on the torus $\mathbb{T}^2$, where $\mathbb{T}:= \mathbb{R}/\mathbb{Z}$, with expanding base maps of the form $bx \mod 1$, where $b\geq 2$ is an integer. More specifically, we study maps $F: \mathbb{T}^2 \rightarrow \mathbb{T}^2$ given by 
\begin{equation}\label{map}
    F(x,y) = \left(bx, f(x,y) \right),
\end{equation}
where $f: \mathbb{T}^2 \rightarrow \mathbb{T}$, also denoted by $f_x(y)$, is chosen from a certain open class of functions which are $C^1$ in $x$ and $y$. For each $n\in \mathbb{N}$, the $n$th iterate of the map $F$ is denoted by 
    $$F^n(x,y) = (b^n x, f^n_{x}(y)) = (b^n x, f_{b^{n-1}x}\circ \cdots \circ f_x(y)).$$
The orbit of a point $(x, y) \in \mathbb{T}^2$ under $F$ is denoted by $((x_j, y_j))_{j\geq 0}$, where
    \begin{align*}
        (x_j, y_j) & := F^j(x,y).
    \end{align*}
    
Let us define the fibered Lyapunov exponent at a point $(x,y)$ to be 
    \begin{align*}
        L(x,y) 
        & := \limsup_{n\rightarrow \infty} \ \frac{1}{n}\log \Big| \frac{\partial y_n}{\partial y} \Big| \\
        & \ = \limsup_{n\rightarrow \infty} \ \frac{1}{n} \sum_{j=0}^{n-1}  \log \Big|\partial_{y} f_{x_j}(y_j) \Big|. 
        \stepcounter{equation}\tag{\theequation} \label{lyapunov}
    \end{align*}
The Lyapunov exponent $L(x,y)$ thus captures the stability of the orbit of $(x,y)$, along the direction of the fibre. We are interested in finding certain conditions on $b$ and $f$ for which we have $L(x,y)<0$ (almost everywhere). In Theorem \ref{theorem}, we prove that for each fixed $f$ that satisfies the appropriate conditions the Lyapunov exponent $L(x,y)$ has a uniform negative bound, for a.e. $(x,y)$. Let us now briefly discuss the context of this question and some of the previous results.

The classic result of Oseledets' multiplicative ergodic theorem proves that, for a well behaved linear cocycle $F(x,y)$, there exists a filtration of the tangent space into subspaces such that each subspace has a distinct Lyapunov exponent. There has been a continued interest to sharply estimate Lyapunov exponents for classic systems such as the H\'{e}non map and Schr\"{o}dinger cocycles. Since the exponents capture asymptotic expansion or contraction in the tangent space, the systems that have non-zero Lyapunov exponents almost everywhere exhibit a weaker form of hyperbolicity called the non-uniform hyperbolicity. In \cite{Young:1993}, L.S.Young established an open class of $SL(2,\mathbb{R})$-cocycles that have a positive Lyapunov exponent. Various techniques have since been established to approach this question for different cocycles and we refer to \cite{Barreira:2006} for an overview. For a dynamical system with negative Lyapunov exponents, as a natural extension, one can investigate the existence of attractors, SRB measures and their corresponding regularity. In \cite{Tanzi_Young:2020}, Tanzi and Young have proved non-uniform hyperbolicity results and the existence of SRB measures for certain skew-product maps on $\mathbb{T}^2 \times \mathbb{T}$. We also refer to their insightful discussion on systems with different types of coupling of hyperbolic maps and diffeomorphisms on the circle, and systems with attractors that have wild geometric structure. Let us now discuss a few results that are closer to our setup. 

In \cite{Homburg:2014}, Homburg uses a numerical experiment on the map 
$$F(x,y)=\left(3x, x+y+\frac{1}{8}\sin(2\pi y)\right)$$
to demonstrate the existence of synchronisation in certain skew-product maps by showing a fast global contraction on its fibres. The main result of \cite{Homburg:2014} proves the existence of an open class of circle diffeomorphisms forced by expanding base maps for which almost every fibre exhibits a global contraction and the maps are topologically mixing. In Theorem \ref{theorem}, we establish an open class of functions such that for a sufficiently expanding base map, there is local contraction on each fibre; moreover, the rate of this contraction is uniformly bounded by a constant that depends on selective characteristics of $f$. 

The method used in this paper, of using a strongly expanding base map to study dynamics of non-monotonic functions $f(x,y)$, is motivated from \cite{Bjerklov:2018} and \cite{Bjerklov:2020} by Bjerkl\"{o}v.  In \cite{Bjerklov:2018}, the main result is the global synchronisation of orbits, for skew-product maps with functions $f(x,y)$ which are monotonous in $x$ and $y$. For the functions $f(x,y)$ considered, there exists a region in $\mathbb{T}^2$ where there is strong contraction on the fibre, and another region with strong expansion. These techniques were further inspired by \cite{Viana:1997}, where Viana introduced the concept of `admissible curves' to establish an open class of maps that have multiple positive Lyapunov exponents, which then has implications on the stability of its attractors. A similar construction can also be found in \cite{Tsujii:2001}, where Tsujii proved the existence of absolutely continuous SRB measures for certain skew-product maps on $\mathbb{T} \times \mathbb{R}$. A similar notion of admissible curves has been used by Castorrini and Liverani in \cite{Castorrini_Liverani:2020}, for a class of partially hyperbolic systems that were proved to have finitely many SRB measures. It is worth noting that various techniques have been developed to analyse non-uniform hyperbolicity by using monotonicity conditions imposed on the systems. In \cite{Bjerklov:2022},  Bjerkl\"{o}v obtained precise estimates of the Lyapunov exponents for a specific class of functions $f(x,y) = x+g(y)$, where $g$ is an orientation preserving $C^2$ diffeomorphism. These results were further generalised for a broader class of monotonic functions, by Bjerkl\"{o}v and Krikorian, in \cite{Bjerklov_Krikorian:2024}. Some more relevant results will be discussed in Subsections \ref{remarks} and \ref{examples}, in the context of Theorem \ref{theorem}.


\begin{figure}
   \includegraphics[width=1.02\linewidth, center]{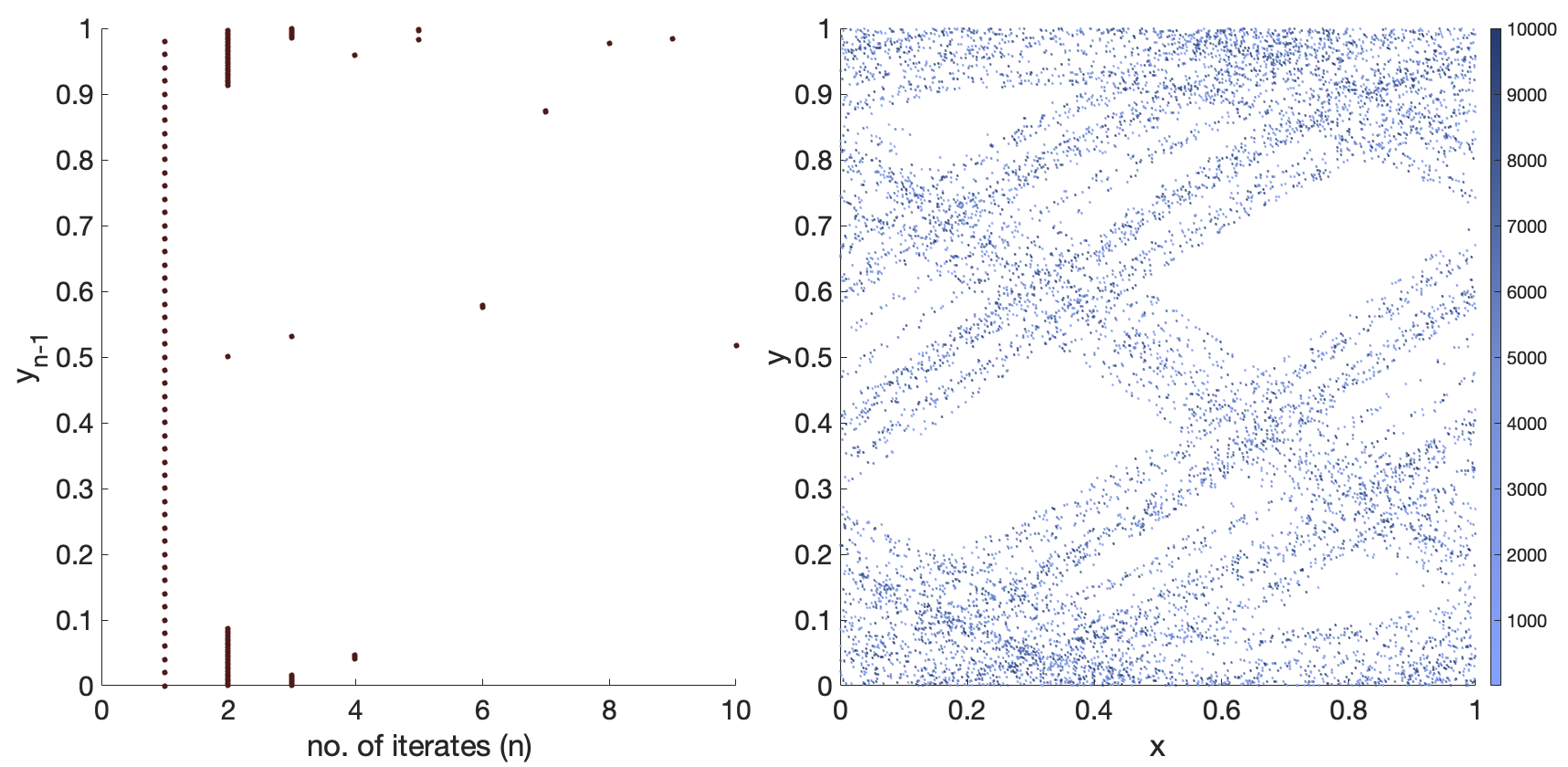}
\caption{The left panel shows 50 distinct points on the fibre $x=\frac{1}{350 \pi}$ and their iterates under the map $F_{1/6000}$, defined in Example 1 of Subsection \ref{examples}, which is of the form $(x,y)\mapsto (7x, g(x)+h_{1/6000}(y))$.
The right panel shows the plot of one hundred thousand points in the orbit of the point $(\frac{1}{350 \pi}, \frac{3}{50})$ with respect to the map $F_{1/6000}$.}
\label{graph 1}
\end{figure}



\subsection{Statement of the theorem}

\begin{definition}\label{definition f}
Let $\mathcal{F}(\mathbb{T}^2)$ denote the class of $C^1$ functions $f:\mathbb{T}^2 \rightarrow \mathbb{T}$ with $\|f\|_{C^1} \geq 1$ and that satisfy the following conditions:

\begin{enumerate}
    \item[(a)]\label{a} There exists a set $G \subset \mathbb{T}$ consisting of a finite union of open intervals and a constant $0<C<1$ such that for all $(x,y) \in \mathbb{T}\times G$ we have $|\partial_{y}f_x(y)| < C < 1.$
    \item[(b)]\label{b} There exists an integer $s \geq 1$ and a constant $0 < \varepsilon < \frac{\log (1/C)}{6\ \log (\|f\|_{C^1}/C) }$ such that for any $y \in \mathbb{T}$ the set $\{x\in \mathbb{T}:\ f_x(y) \notin G\}$ consists of at most $s$ intervals and has measure at most $\varepsilon$.
\end{enumerate}
\end{definition}

Note that $\mathcal{F}(\mathbb{T}^2)$ is open in the $C^1$ topology. In the above definition, the lack of any monotonicity condition in $x$ is a major distinction between the result in Theorem \ref{definition f} and the previous results in \cite{Bjerklov:2018},\cite{Bjerklov:2022}, and \cite{Young:1993}.

\begin{theorem}\label{theorem}
For each $f \in \mathcal{F}(\mathbb{T}^2)$ there exists $b_0 = b_0(\|f\|_{C^1},C,s,\varepsilon) > 2$ such that for all $b > b_0$ its corresponding map $F$, as defined in equation (\ref{map}), has the following estimate:
$$L(x,y) < \frac{\log C}{2}, \mbox{ for a.e.  $(x,y) \in \mathbb{T}^2$}.$$
\end{theorem}

\subsection{Remarks}\label{remarks}
\begin{itemize}
        \item For a map $F$ that exhibits local contraction of orbits on the fibre, the next natural question is about the global behaviour of the orbits. More specifically, we can study the structure and dynamics of invariant sets in $\mathbb{T}^2$ that attract a set of positive measure.
        \item Since $F$ has an expanding base map, it is not invertible. A classic construction to study such a map is to extend the base map to a homeomorphism on a solenoid, as in \cite{Homburg:2014}. Let us define the shift space $\Omega:= \{\omega \in \mathbb{T}^{\mathbb{Z}}: b \omega_i = \omega_{i+1} \}$ and consider the map $\bar{F}: \Omega \times \mathbb{T} \rightarrow \Omega \times \mathbb{T}$, given by $\bar{F}(\omega,y) = (\sigma(\omega), f(\omega_0,y))$, where $\sigma$ is the left shift map. Due to previous results by Le Jan \cite{Jan:1987}, the map $\bar{F}$ must have finitely many minimal invariant sets in the solenoid space. The projection of the invariant sets of $\bar{F}$ to $\mathbb{T}^2$ gives us the invariant sets of $F$. 
        \item For a typical non-monotonic function $f \in \mathcal{F}(\mathbb{T}^2)$, we expect the invariant sets of $F$ to be similar to the fat solenoid attractors in \cite{Tsujii:2001}. In Figure \ref{graph 1}, the plot of a hundred thousand points of an orbit contains hints of regular curves, which are likely the projection of invariant curves in the solenoid space $\Omega \times \mathbb{T}$.
        \item For a map $F$ with fat solenoid attractors, one could also further investigate the existence of absolutely continuous SRB measures supported on the attractors. We refer to \cite{Young:2002} for a discussion on SRB measures.
    \end{itemize}


\subsection{Examples}\label{examples}
Let us illustrate the result of Theorem \ref{theorem} with a simple and then a classic example.

\subsubsection{\textbf{Example 1:}} 
For each $0 < \tilde{\varepsilon} < 1/8$, let us consider a circle diffeomorphism $h_{\tilde{\varepsilon}}: \mathbb{T} \rightarrow \mathbb{T}$ such that $h_{\tilde{\varepsilon}}(0) = 0$ and $h'_{\tilde{\varepsilon}}(y) \leq 0.18$, for all $y \in \mathbb{T} \setminus [0.5-2\tilde{\varepsilon}, 0.5+2\tilde{\varepsilon}]$. For simplicity, let us consider functions $h_{\tilde{\varepsilon}}(y)$ that are monotonic on $\mathbb{T}$, linear on the arcs $(0.5-\tilde{\varepsilon}, 0.5+\tilde{\varepsilon})$ and $(0.5+2\tilde{\varepsilon}, 0.5-2\tilde{\varepsilon})$ on the torus, and quadratic on the rest of the torus. Then, let us consider a skew-product map $F_{\tilde{\varepsilon}}:\mathbb{T}^2 \rightarrow \mathbb{T}^2$ given by 
\begin{align*}
    F_{\tilde{\varepsilon}}(x,y) = \left( 7x, g(x)+h_{\tilde{\varepsilon}}(y) \right),
\end{align*}
where $g(x) = (1-\cos(3 \pi x))/2$ is a non-monotonic $C^1$ function on $\mathbb{T}$. For each $0<\tilde{\varepsilon}<1/8$, the $C^1$ function $g(x)+h_{\tilde{\varepsilon}}(y)$ satisfies the condition in Definition \ref{definition f}(a), where we have the open set $\tilde{G}_{\tilde{\varepsilon}} = \mathbb{T} \setminus [0.5-2\tilde{\varepsilon}, 0.5+2\tilde{\varepsilon}]$ and the rate of contraction in the region $\mathbb{T} \times \tilde{G}_{\tilde{\varepsilon}}$ to be $0.18<C<1$. 

For $\tilde{\varepsilon}=1/20$, the Figure \ref{graph 2} illustrates the intersection of the curve $y = g(x)+0.5$ with the region $\mathbb{T} \times \tilde{G}_{\tilde{\varepsilon}}$. Note that the measure of the set $\left\{x: g(x)+0.5 \notin \tilde{G}_{\tilde{\varepsilon}} \right\}$ decreases for smaller choices of $\tilde{\varepsilon}$. For a sufficiently small $\tilde{\varepsilon}$, we can show that the function $g(x)+h_{\tilde{\varepsilon}}(y)$ satisfies the condition in Definition \ref{definition f}(b). Then, by Theorem \ref{theorem}, we will observe convergence of orbits of arbitrarily close points on almost every fibre. In Figure \ref{graph 1}, we consider equidistributed points on a fixed fibre and plot their respective orbits with respect to $F_{1/6000}$. Under iteration, we observe that they synchronise to a single orbit. 

Similarly, for a map of the form $F(x,y) = (bx, g(x)+h(y))$, where $g(x)+h(y)$ is in $\mathcal{F}(\mathbb{T}^2)$, $h(y)$ has multiple attracting regions that satisfy an appropriate symmetry, and $g(x)$ has a sufficiently small norm, we can observe convergence of almost every point on a fibre to finitely many orbits. 


\begin{figure}
    \centering
    \includegraphics[width=0.95\linewidth]{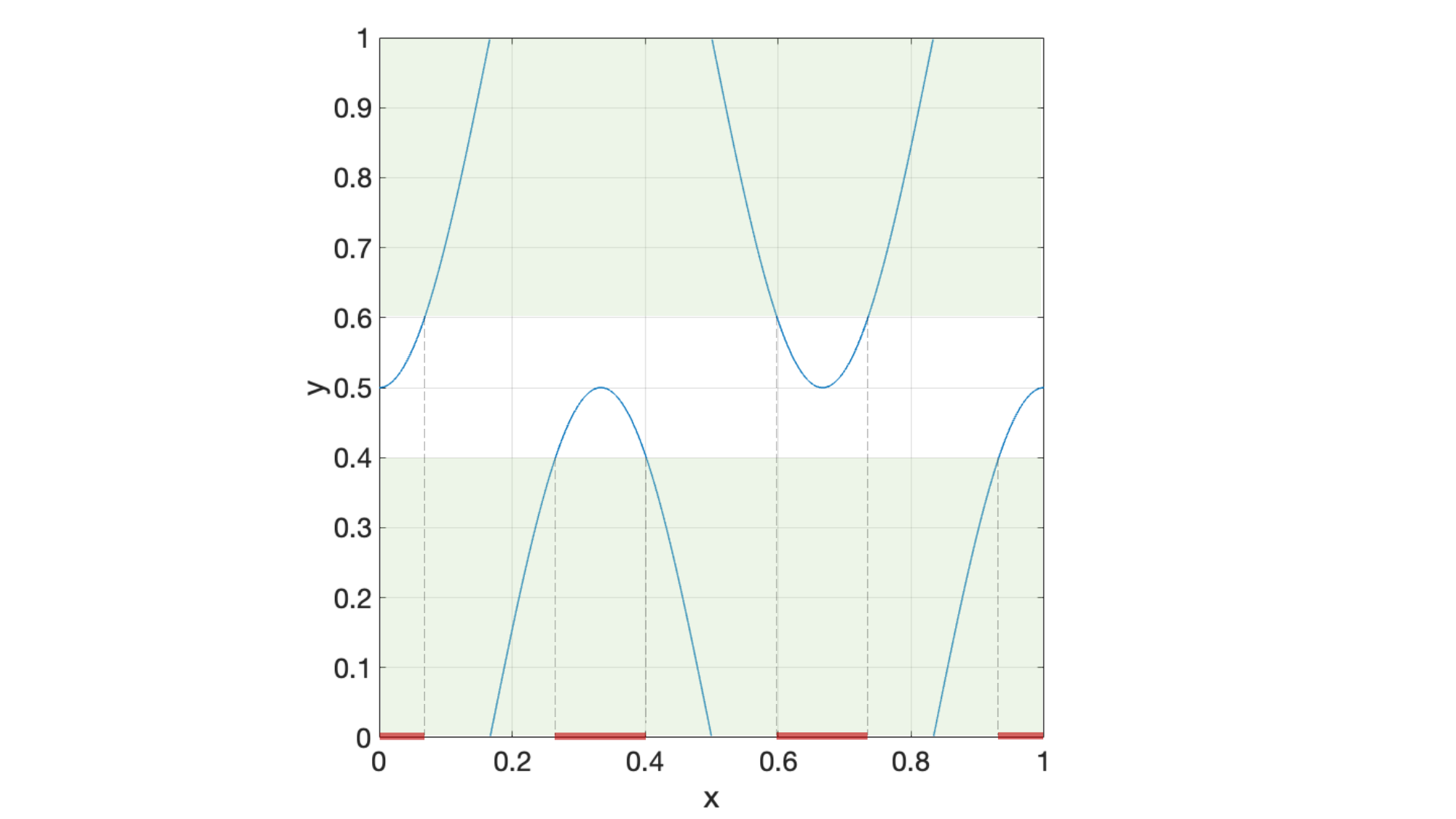}
    \caption{For the skew-product map $F_{1/20}(x,y)$, as defined in Example 1, the plot above illustrates the condition in Definition \ref{definition f}(b). Here, the curve is $y = g(x)+0.5$, the shaded region is $\mathbb{T} \times \tilde{G}_{1/20} = \mathbb{T} \times \left([0, 0.4) \cup (0.6,1] \right)$, and the set $\left\{x\in \mathbb{T}: g(x)+0.5 \notin \tilde{G}_{1/20}\right\}$ is marked in bold on the $x$ axis.} 
    \label{graph 2}
\end{figure}



\subsubsection{\textbf{Example 2:}}
 Let us look at, a well-known example, the Schr\"{o}dinger cocycle denoted by $H: \mathbb{T} \times \mathbb{R}^2 \rightarrow \mathbb{T}\times \mathbb{R}^2$ and defined as 
 \begin{align*}
     H(x,\tilde{y}) = \left( bx, \begin{bmatrix}
0 & 1 \\
-1 & \lambda v(x) - E 
\end{bmatrix} \tilde{y} \right) = (bx, A(x)\tilde{y}), 
 \end{align*}
where $\lambda > 0$, $E \in \mathbb{R}$ and $v(x)$ is a certain non-constant $C^1$ function. Analogous to the condition in Definition \ref{definition f}(b), we consider specific conditions on $v(x)$. We assume that there exist constants $\varepsilon_0 > 0$, $\beta > 0$, and an integer $s \geq 1$ such that, for every $a \in \mathbb{R}$ and $\varepsilon < \varepsilon_0$, the set $\{x\in \mathbb{T}: |v(x) - a| < \varepsilon \}$ consists of at most $s$ intervals, and each interval is of length at most $\varepsilon^\beta$. This condition is satisfied by any non-constant real analytic function $v(x)$. Note that $A(x) \in SL(2, \mathbb{R})$ for all $x\in \mathbb{T}$. Let us denote the iterates of the cocycle $H$ by $H^n(x,\tilde{y}) = (b^nx, A^n(x)\tilde{y})$, where $A^n(x):= A(b^{n-1}x) \cdots A(x)$. By Kingman's subadditive ergodic theorem and ergodicity of the base map, we have following limit:
\begin{align}\label{max lyap}
  \tilde{L} = \lim_{n \rightarrow \infty} \frac{1}{n} \log \|A^{n}(x) \|, \ \ \text{ for a.e. } x\in\mathbb{T},
\end{align}
 and the constant $\tilde{L}$ is defined as the maximal Lyapunov exponent of the Schr\"{o}dinger cocycle $H$. 
 
 The Schr\"{o}dinger cocycle induces a non-linear cocycle on the projective space, denoted by $\tilde{F}:\mathbb{T} \times S \rightarrow \mathbb{T}\times S$ where $S = [-\pi/2,\pi/2]/\sim$. It is given by
\begin{equation}\label{ex 1 defn F}
    \tilde{F}(x,y) = (bx, \tilde{f}_x(y)) = \big(bx, \arctan\big(\lambda v(x) - E - 1/\tan y \big)  \big).
\end{equation}
From equation (\ref{lyapunov}), the fibered Lyapunov exponent at $(x,y)\in \mathbb{T}\times S$ is given by
\begin{equation}\label{fiber lyap schro}
    L(x,y) = -2\limsup_{n\rightarrow \infty} \frac{1}{n} \log \left(\prod_{i=0}^{n-1} |\tan(y_i)| \right).
\end{equation}

Since the action of $A(x) \in SL(2,\mathbb{R})$ is area preserving, the rate of the convergence of the angle in equation (\ref{fiber lyap schro}), and the rate of increase of the norm in equation (\ref{max lyap}) are related such that we have
 \begin{align}\label{ex1lyapunov}
     L(x,y) = -2\tilde{L} = -2\limsup_{n\rightarrow \infty} \frac{1}{n} \log \|A^{n}(x) \|, \ \ \text{ for a.e. } (x,y)\in \mathbb{T}\times S.
 \end{align}
In equation \eqref{ex 1 defn F}, the function $\tilde{f}_x(y)$ is $C^1$ for all $x\in \mathbb{T}$ and for all $y\in S\setminus \{0\}$. For fixed $\lambda$, $E$, and $v(x)$, $|\partial_x \tilde{f}_x(y)|$ is uniformly bounded in $\mathbb{T} \times S$ and $|\partial_y \tilde{f}_x(y)|$ is uniformly bounded in $\mathbb{T}\times (S\setminus \{0\})$. Let us define the set $G:= \left\{y\in S^1: |\tan y|> \sqrt{\lambda}\right\}$, which is an open interval. For the fixed $v(x)$, there exists constants $s$ and $\beta$ such that for any $y\in S$, the set $ \{x:\tilde{f}_x(y)\notin G\} = \left\{x: \left|\lambda v(x)-E-1/\tan y\right|< \sqrt{\lambda}\right\}$ consists of at most $s$ intervals, and each interval has length at most $\lambda^{-\beta/2}$. Thus, for a big enough $\lambda$, the Schr\"{o}dinger cocycle is in the realm of skew-product maps discussed in Theorem $\ref{theorem}$ and, by equations (\ref{max lyap}) and (\ref{ex1lyapunov}), the Schr\"{o}dinger cocycle must have a negative Lyapunov exponent $L(x,y)$, for a.e. $(x,y) \in \mathbb{T} \times S$. This is consistent with the main result in \cite{Bjerklov:2020} where, for a sufficiently big $b$ and $\lambda$, the maximal Lyapunov exponent $\tilde{L}$ is positive and uniformly bounded below.

\section{Preliminaries}

\subsection{Notation}\label{notation} 
By definition, $\mathbb{T} = [0,1]/ \sim$, and thus there is a natural correspondence between constructions on $\mathbb{T}$ and those on $[0,1)$. Let $d$ denote the natural metric endowed on the torus $\mathbb{T}$. Let $\pi_2: \mathbb{T}^2 \rightarrow \mathbb{T}$ denote the projection of the second coordinate, given by $\pi_2(x,y) = y$ for all $(x,y)\in \mathbb{T}^2$.

\subsection{Counting lemma}\label{counting lemma}
For a fixed $f \in \mathcal{F}(\mathbb{T}^2)$, let $R := \|f\|_{C^1}$ and $C$ be constants corresponding to $f$, as mentioned in Definition \ref{definition f}(a). Since we have $0<C<1$, we can define a constant  $l :=  \frac{ \log(R)}{\log(1/C)} $. Note that, the fixed $l \geq 0$ is large enough such that $RC^l \leq 1$. Recall that, for the fixed function $f$, there exists an open set $G$ as defined in Definition \ref{definition f}(a).

\begin{lemma}\label{counting}
Given a point $(x_0,y_0) \in \mathbb{T}^2$ and an integer $N \geq 0$, assume that there exists $0 \leq p < \frac{1}{2(l+1)}$ such that $[Np]$ is the number of indices $j \in \{0, \dots, N-1 \}$ for which $y_{j} \notin G$. Then, $$\prod_{j=0}^{N-1} |\partial_{y} f(x_j, y_j)| < C^{N/2}.$$
\end{lemma}

\begin{proof}
Since $R = \|f\|_{C^1}$, for every $(x_j, y_j)\in \mathbb{T}^2$ we have $|\partial_{y} f(x_j, y_j)| \leq R$. From the conditions in Definition \ref{definition f}(a), if $(x_j, y_j)\in \mathbb{T} \times G$ then we have $|\partial_{y} f(x_j, y_j)|< C < 1$. Since $[Np]$ is the number of indices $j \in \{0, \dots, N-1 \}$ for which $(x_j,y_j)\notin \mathbb{T} \times G$, we have
\begin{align*}
    \prod_{j=0}^{N-1} |\partial_{y} f(x_j, y_j)| & \leq R^{[Np]} C^{N-[Np]} \\  
    & \leq (RC^{l})^{[Np]} C^{N-[Np](l+1)}.
\end{align*}
Since the constant $l$ is chosen such that $RC^l \leq 1$, we have $(RC^l)^{[Np]} \leq 1$. From the assumption on $p$, we have $[Np] < \frac{N}{2(l+1)}$ and thus $N-[Np](l+1)> \frac{N}{2}$. Since $0<C<1$, we obtain the result:
\begin{align*}
\prod_{j=0}^{N-1} |\partial_{y} f(x_j, y_j)| < C^{N/2}.
\end{align*}

\end{proof}


\subsection{Probability Results}\label{probability}

For a fixed $b\geq 2$, let us fix some notations and constructions. The point $0 \in \mathbb{T}$ has exactly $b$ distinct pre-images under the map $T(x) = bx\ (\mod 1)$. Let us denote those pre-images by $0=\theta_0 < \theta_1 < \cdots < \theta_b$, where $\theta_b = 0\ (\mod 1)$. For each $j \in \{1,\dots ,b\}$, we define a set $I_j:= [\theta_{j-1},\theta_j) \subset \mathbb{T}$. Then, $\mathcal{P}_1 := \left\{I_j\right\}_{j=1}^{b}$ is a partition of $\mathbb{T} = [0,1]/ \sim$. For each $j \in \{1,\dots ,b\}$, let $T^{-1}_j$ denote the branch of the inverse map $T^{-1}$ which maps $\mathbb{T}$ onto $I_j$. For each $n\geq 1$ and $j_1,\dots, j_n \in \{1,\dots ,b\}$, we define a connected set 
\begin{align*}
    I_{j_1 \cdots j_n}:= I_{j_1}\cap T^{-1}(I_{j_2})\cap \cdots \cap T^{-n+1}(I_{j_n}).
\end{align*}
Then, for each integer $n\geq 1$, the set 
\begin{align*}
    \mathcal{P}_n := \left\{I_{j_1 \cdots j_n}: j_1,\dots, j_n \in \{1,\dots ,b\} \right\}
\end{align*}
is a partition of $\mathbb{T}$.

For ease of visualisation, note that the torus $\mathbb{T}$ is topologically equivalent to $[0,1]/ \sim$. Thus, for each $j \in \{1,\dots ,b\}$, the connected set $I_{j} \subset \mathbb{T}$ corresponds to the interval $\left[\frac{j-1}{b}, \frac{j}{b}\right) \subset [0,1)$, and the map $T_j^{-1}$ corresponds to the function $\frac{x+j-1}{b}$, where $x\in [0,1)$. Note that, for each $n$, the partition $\mathcal{P}_n$ is a refinement of $\mathcal{P}_{n-1}$. Further, each interval $I_{j_1 \cdots j_n}$ in partition $\mathcal{P}_n$ has a corresponding sequence of nested intervals given by
\begin{align}\label{nested I_js}
    I_{j_1} \supset I_{j_1 j_2} \supset \cdots \supset I_{j_1 \cdots  j_n}.
\end{align}

Let $y_0 \in \mathbb{T}$ be fixed and, with respect to $y_0$, let us characterise intervals in the partitions $(\mathcal{P}_n)_{n\geq 1}$ as either good or bad. For each $n\geq 1$ and $j_1, \dots ,j_n \in \{1, \dots, b\}$, the interval $I_{j_1 \cdots j_n}$ is called a good interval if, for every $x \in I_{j_1 \cdots j_n}$ the point $y_n = \pi_2 \circ F^{n}(x, y_0)$ belongs in the set $G$; however, if the interval $I_{j_1 \cdots j_n}$ is not good then it called bad. Equivalently, $I_{j_1 \cdots j_n}$ is bad if there exists $x \in I_{j_1 \cdots j_n}$ such that $y_n = \pi_2 \circ F^{n}(x, y_0)$ does not belong in $G$. Recall that $\pi_2$ denotes the projection on the second coordinate.

The following elementary probability results from \cite{Bjerklov:2020}, along with Lemma \ref{counting}, reduces our goal of estimating the rate of contraction along the fibre, for almost every orbit, to a problem of counting bad intervals. 

\begin{lemma}\label{prob1}
Let $1 \leq q \leq b$ be an integer. Assume that $q$ of the intervals in $\{I_{j}\}_{j=1}^{b}$ are bad. Furthermore, for each $n\geq 1$, $j_1,\dots ,j_n \in \{1,\dots , b\}$, and its corresponding interval $I_{j_1\cdots j_n}$, assume that $q$ of the intervals in $\{I_{j_1\cdots j_n j}\}_{j=1}^{b}$ are bad. Then, for each $n \geq 1 $ and $0\leq m \leq n$, the set
\begin{align*}
    \left\{x\in \mathbb{T}: \mbox{$x\in I_{j_1\cdots j_n}$ and exactly $m$ of the intervals $I_{j_1}, I_{j_1 j_2}, \dots, I_{j_1\cdots j_n}$ are bad } \right\}
\end{align*}
has measure $\binom{n}{m}\frac{q^m(b-q)^{n-m}}{b^n}$.
\end{lemma}
If $q$ of the intervals in $\{I_{j_1\cdots j_{n-1}j}\}_{j=1}^{b}$ are bad, then each $I_{j_1\cdots j_n} \in \{I_{j_1\cdots j_{n-1}j}\}_{j=1}^{b}$ can be interpreted as a Bernoulli trial, where the probability of $I_{j_1\cdots j_n}$ being bad is $q/b$ and the probability of it being good is $1-q/b$. If there exists $1\leq q \leq b$ that satisfies the given assumption in Lemma \ref{prob1} for all $n\geq 1$, then each $I_{j_1\cdots j_n}$ can be interpreted as a sequence of $n$ Bernoulli trials, corresponding to the intervals in equation (\ref{nested I_js}), and we thus obtain the binomial expansion in Lemma \ref{prob1}.

\begin{lemma}\label{prob2}
Under the same hypothesis, for $n \in \mathbb{N}$, let us define a sequence of sets
\begin{multline*}
    M_n := 
    \{x\in \mathbb{T}: \mbox{ $x\in I_{j_1\cdots j_n}$ and at most $[2(q/b)n]$ of the intervals}\\
        \mbox{$I_{j_1}, I_{j_1 j_2}, \dots, I_{j_1\cdots j_n}$ are bad } \}.
\end{multline*}
Then, the set
\begin{equation*}
   M = \liminf_{n\rightarrow \infty} M_n = \bigcap_{N=1}^{\infty}\bigcup_{n=N}^{\infty} M_n,
\end{equation*}
has measure 1.
\end{lemma}

\begin{proof}
From Lemma \ref{prob1} and the definition of set $M_n$, we have 
\begin{align*}
    |M_n| = \sum_{m\leq [2(q/b)n]} \binom{n}{m} \frac{q^m (b-q)^{n-m}}{b^n}.
\end{align*}
Thus, the complement of the set $M_n$ satisfies the following:
\begin{align}\label{count}
    |\mathbb{T} \setminus M_n| = \sum_{m > [2(q/b)n]} \binom{n}{m} \frac{q^m (b-q)^{n-m}}{b^n}.
\end{align}

Note that the expression in equation (\ref{count}) is the tail of the binomial distribution with parameters $n$ and $q/b$. Thus, by Hoeffding's inequality, the measure $|\mathbb{T} \setminus M_n|$ decays exponentially as $n \rightarrow \infty$. By the Borel-Cantelli lemma, the set $\limsup_{n \rightarrow \infty} (\mathbb{T}\setminus M_n)$ has measure zero. Thus, the set $M$ has measure 1. 
\end{proof}


\subsection{Basic estimates for \texorpdfstring{$f$}{f}}

For each $f \in \mathcal{F}(\mathbb{T}^2)$, from the Definition $\ref{definition f}$(b), there exist corresponding constants $s\geq 1$ and $\varepsilon>0$ that provide the measure theoretic and topological estimates of the set $\{x\in \mathbb{T}: f_x(y) \notin G\}$, for any $y \in \mathbb{T}$. Let $\phi: \mathbb{T} \rightarrow \mathbb{T}$ be a function, with a corresponding point $y \in \mathbb{T}$ such that $|\phi(x)-y|$ is sufficiently small for all $x\in \mathbb{T}$. In this subsection, we obtain measure theoretic and topological estimates, similar to those in Definition \ref{definition f}(b), for sets of the form $\{x\in \mathbb{T}: f_x(\phi(x)) \notin G\}$. This is later used in the proof of Lemma \ref{proof1}. Now we establish the necessary parameters and  estimates. 

\begin{lemma}\label{deltay}
Let $\varepsilon < \varepsilon'<1$ be a constant. For any fixed $y_0 \in \mathbb{T}$, there exists $\delta_{y_0} = \delta_{y_0}(\varepsilon, \varepsilon') > 0$ and $s_{y_0} \geq 1$ such that the set 
$$\Big\{x\in \mathbb{T}: f_x(y)\notin G \text{ for some } y\in B_{\delta_{y_0}}(y_0) \Big\}$$
can be covered by a set that consists of at most $s_{y_0} \leq s$ intervals and that has measure at most $\varepsilon'$.
\end{lemma}

\begin{proof}
Let us fix $y_0 \in \mathbb{T}$, define a constant $\gamma:= \frac{1}{4s} (\varepsilon' - \varepsilon) > 0$, and define a set $B:= \Big\{ x : f_{x}(y_0) \notin G \Big\}$. Recall from Definition \ref{definition f}(b) that $B$ consists of at most $s$ intervals and measure $\varepsilon$. The lemma is trivially true when $B = \emptyset$ and so, let us assume that $B \neq \emptyset$. Let $1 \leq m \leq s$ denote the number of intervals in the set $\mathbb{T}\setminus B$, which will thus be of the form
\begin{align*}
    \mathbb{T} \setminus B = \left\{x\in \mathbb{T}: f_x(y_0) \in G \right\} = \bigcup_{i=1}^{m} (a_i, b_i).
\end{align*}
Since $G$ is open and $|\partial_xf_x(y_0)|<R$ for all $x\in \mathbb{T}$, for each interval $(a_i, b_i)$ we can choose a sufficiently small $r_i>0$ such that
\begin{align*}
    (a_i+\gamma, b_i-\gamma) \subset \left\{x\in (a_i, b_i): (f_x(y_0)-r_i, f_x(y_0)+r_i) \subset G\right\}.
\end{align*}
Since $\mathbb{T}\setminus B$ consists of finitely many intervals, for a sufficiently small $r>0$ we have 
\begin{align}\label{deltay eqn 1}
    \bigcup_{i=1}^{m}(a_i+\gamma, b_i-\gamma) \subset \left\{x\in \mathbb{T} \setminus B: (f_x(y_0)-r, f_x(y_0)+r) \subset G\right\}.
\end{align}
Since $f(x,y)$ is a $C^1$ map on a compact set, there exists $\delta_{y_0}>0$ such that if $y\in B_{\delta_{y_0}}(y_0)$ then $|f_x(y)-f_x(y_0)|<r$, for all $x\in \mathbb{T}$. Thus, for any $y \in B_{\delta_{y_0}}(y_0)$ and $x\in \mathbb{T}$, we have $f_x(y) \in (f_x(y_0)-r, f_x(y_0)+r)$ and this gives us the following set relation:
\begin{multline}\label{deltay eqn 2}
  \left\{x\in \mathbb{T} \setminus B: (f_x(y_0)-r, f_x(y_0)+r) \subset G \right\} \\
  \subset \left\{x\in \mathbb{T}\setminus B : \mbox{ $f_x(y)\in G$ for all $y\in B_{\delta_{y_0}}(y_0)$ } \right\}. 
\end{multline}
The complement of the set on the right-hand side of equation (\ref{deltay eqn 2}) is given by 
\begin{multline}\label{deltay eqn 3}
    \mathbb{T} \setminus \left\{x\in \mathbb{T}\setminus B : \mbox{$f_x(y)\in G,$ for all $y\in B_{\delta_{y_0}}(y_0)$} \right\} \\
    = B \cup \left\{ x\in \mathbb{T}: \mbox{ $f_{x}(y) \notin G$ for some $y\in B_{\delta_{y_0}}(y_0)$} \right\}.
\end{multline}
Thus, from equation (\ref{deltay eqn 1}), (\ref{deltay eqn 2}) and (\ref{deltay eqn 3}), we obtain the following: 
\begin{align*}
    \left\{ x\in \mathbb{T}: \mbox{ $f_{x}(y) \notin G$ for some $y\in B_{\delta_{y_0}}(y_0)$} \right\} \subset \bigcap_{i=1}^{m} \left(\mathbb{T} \setminus (a_i+\gamma, b_i-\gamma)\right).
\end{align*}
The set $\bigcap_{i=1}^{m} \left(\mathbb{T} \setminus (a_i+\gamma, b_i-\gamma)\right)$ has finitely many intervals, which is denoted by $s_{y_0}\leq s$, and  has measure at most $\varepsilon+2\gamma s$. For the specific choice of $\gamma$, the set
$\{ x\in \mathbb{T}: \mbox{ $f_{x}(y) \notin G$ for some } y\in B_{\delta_{y_0}}(y_0) \}$ has measure at most $\varepsilon'$ and this proves the lemma.

\end{proof}

\begin{lemma}\label{delta}
For each $\varepsilon < \varepsilon' < 1$, there exists $\delta > 0$ such that, for any $y_0 \in \mathbb{T}$, the set 
$$ 
\Big\{ x\in \mathbb{T}: f_x(y) \notin G, \mbox{ for some } y\in B_{\delta}(y_0) \Big\}
$$
can be covered by a set consisting of at most $s$ intervals and that has measure at most $\varepsilon'$.
\end{lemma}

\begin{proof}

Let us fix a constant $\varepsilon < \varepsilon' < 1$. For each $y\in \mathbb{T}$, there exist constants $\delta_{y}>0$ and $s_{y}$ as defined in Lemma \ref{deltay}. Let us consider the collection $\big\{ B_{\delta_{y}}(y)\big\}_{y \in \mathbb{T}}$ that forms an open cover for $\mathbb{T}$. Since $\mathbb{T}$ is compact, there exists a finite subcover denoted by $\big\{ B_{\delta_{y_i}}(y_i)\big\}_{i=1}^N$. Let us define an index set
$$
J:= \left\{(i,j): \mbox{$i,j \in \{1,\dots ,N\}$ and $B_{\delta_{y_i}}(y_i) \cap B_{\delta_{y_j}}(y_j) \neq \phi$} \right\}.
$$
For each $(i,j) \in J$, let $r_{ij}$ denote the radius of the intersection $B_{\delta_{y_i}}(y_i) \cap B_{\delta_{y_j}}(y_j)$, and let us choose $\delta = \min \{r_{ij}: (i,j)\in J \}$. By the choice of $\delta$, for each $y_0 \in \mathbb{T}$, there exists an open set $B_{\delta_{y_j}}(y_j)$ in the finite cover such that $B_{\delta}(y_0) \subset B_{\delta_{y_j}}(y_j)$, and thus
\begin{align}\label{delta eqn 1}
    \left\{x\in \mathbb{T}: \mbox{$f_{x}(y) \notin G$ for some $y\in B_{\delta}(y_0)$} \right\} \subset  \left\{x\in \mathbb{T}: \mbox{$f_{x}(y) \notin G$ for some $y\in B_{\delta_{y_j}}(y_j)$} \right\}.
\end{align}
Recall that, we have $s_{y_j} \leq s$ for all $j\in \{1,\dots, N\}$.
From the choice of $\delta_{y_j}$, in Lemma \ref{deltay}, the set $\left\{x\in \mathbb{T}: \mbox{$f_{x}(y) \notin G$ for some $y\in B_{\delta_{y_j}}(y_j)$} \right\}$ can be covered by a set consisting of at most $s$ intervals and has measure at most $\varepsilon'$. Thus, from the equation (\ref{delta eqn 1}), the set $\left\{x\in \mathbb{T}: \mbox{$f_{x}(y) \notin G$ for some $y\in B_{\delta}(y_0)$} \right\}$ can be covered by a set consisting of at most $s$ intervals and has measure at most $\varepsilon'$.

\end{proof}


\section{Additional Construction}

For a fixed $y_0 \in \mathbb{T}$ and for each $j\in \{1,\dots, b\}$, let $\phi_j: \mathbb{T} \rightarrow \mathbb{T}$ be defined as 
\begin{align}\label{definition phi_j}
    \phi_j(x) := \pi_2 \circ F(T_j^{-1}(x),y_0) = f_{T^{-1}_j(x)}(y_0),
\end{align}
where $T^{-1}_j:\mathbb{T} \rightarrow I_j$ is as defined in Subsection \ref{probability}. Further, for each $n>1$ and $j_1,\dots, j_n \in \{1,\dots ,b\}$, we define the function $\phi_{j_1\cdots j_n}: \mathbb{T} \rightarrow \mathbb{T}$ inductively by 
\begin{align*}
    \phi_{j_1\cdots j_n}(x) = \pi_2 \circ F \left( T^{-1}_{j_n}(x), \phi_{j_1\cdots j_{n-1}}\left( T^{-1}_{j_n}(x)\right)\right).
\end{align*}
Since $f \in C^{1}(\mathbb{T}^2)$, for each $n \geq 1$ and $j_1, \dots , j_n \in \{1,\dots ,b\}$, the function $\phi_{j_1\cdots j_n}(x)$ is $C^1$ at every $x\in \mathbb{T} \setminus \{0\}$, but it is not necessarily continuous at $0\in \mathbb{T}$. Thus, the function $\phi_{j_1\cdots j_n}(x)$ is $C^1$ on $\mathbb{T} \setminus \{0\}$. From equation (\ref{definition phi_j}), note that the graph of $\phi_{j}: \mathbb{T}\rightarrow \mathbb{T}$ is the same as the image set of the function $F: I_{j} \times \{y_0\} \rightarrow \mathbb{T}\times \mathbb{T}$. Similarly, the graph of the function $\phi_{j_1\cdots j_n}: \mathbb{T} \rightarrow \mathbb{T}$ is the same as the image set of the function $F^{n}: I_{j_1\cdots j_n} \times \{y_0\} \rightarrow \mathbb{T}\times \mathbb{T}$. Thus, estimating the set of points $x\in I_{j_1\cdots j_n}$ for which $f_{x}^n(y_0) \in G$, it is equivalent to estimating the set of points $x\in \mathbb{T}$ for which $\phi_{j_1 \cdots j_n}(x) \in G$.


\begin{lemma}\label{derivative}
For each $n\geq 1$ and $j_1, \dots, j_n \in \{1, \dots, b\}$, for any $x \in \mathbb{T}\setminus \{0\} $, we have
\begin{align}\label{derivative bound}
    \big| \phi_{j_1 \cdots j_n}'(x) \big| \leq \frac{R}{b}+ \cdots +\frac{R^n}{b^n}.
\end{align}
\end{lemma}

\begin{proof}
For $n=1$ and each $j \in \{1,\dots ,b\}$, the map $\phi_j(x) = f_{T^{-1}_j(x)}(y_0)$ satisfies the bound 
\begin{align*}
    |\phi_j'(x)| \leq |\partial_{x}f_{T^{-1}_j(x)}(y_0)| \leq \frac{R}{b},
\end{align*}
for all $x\in \mathbb{T} \setminus \{0\}$. For an arbitrary $n \geq 2$, let us assume that for each $j_1, \dots ,j_{n-1} \in \{1,\dots ,b\}$ we have the bound in equation (\ref{derivative bound}). For any $j_1,\dots, j_n \in \{1,\dots ,b\}$, since 
$$\phi_{j_1 \cdots j_n}(x) = \pi_2 \circ F\left(T_{j_n}^{-1}(x), \phi_{j_1\cdots j_{n-1}}\left(T_{j_n}^{-1}(x)\right) \right)$$
for every $x\in \mathbb{T}\setminus \{0\}$, we have
\begin{align*}
    \big| \phi_{j_1 \cdots j_{n-1} j_n}'(x) \big| 
    & \leq \Big| \partial_{x} \Big( f_{T_{j_n}^{-1}(x)} \Big( \phi_{j_1 \cdots j_{n-1}}\big(T_{j_n}^{-1}(x)\big) \Big) \Big) \Big| \\ 
    & \leq \|f\|_{C^1}\ |\partial_x(T_{j_n}^{-1}(x))| + \|f\|_{C^1}\ |\phi_{j_1 \cdots j_{n-1}}'(T_{j_n}^{-1}(x))|\ |\partial_x(T_{j_n}^{-1}(x))|.
\end{align*}
Then, using the bounds  $\|f\|_{C_1} \leq R$, $|\partial_x(T_{j_n}^{-1}(x))| = 1/b$, and the bound on $|\phi_{j_1\cdots j_{n-1}}'(x)|$ from the inductive assumption, we prove that for all $x\in \mathbb{T}\setminus \{0\}$,
\begin{align*}
    \big| \phi_{j_1 \cdots j_{n-1} j_n}'(x) \big| 
    & \leq R\ \frac{1}{b} + R\ \Big(\frac{R}{b} + \cdots + \frac{R^{n-1}}{b^{n-1}}\Big)\ \frac{1}{b} \\
    & \leq \frac{R}{b}+ \cdots +\frac{R^n}{b^n}.
\end{align*}
\end{proof}


\section{Proof of the theorem}

For $f \in \mathcal{F}(\mathbb{T}^2)$, let us fix a constant $\varepsilon'$ such that $0 < \varepsilon'< \varepsilon$, where the constant $\varepsilon$ is corresponding to $f$ as defined in Definition \ref{definition f}(b). Recall the constant $l \geq 0$ from Subsection \ref{counting lemma}, defined as $l = \frac{\log(R)}{\log (1/C)}$, that satisfies the inequality $RC^l \leq 1$. Let us fix an arbitrary point $y_0 \in \mathbb{T}$.

 In Lemma \ref{proof1}, we estimate the maximum number of bad intervals in each level of partition of $\mathbb{T}$, with respect to the arbitrarily chosen $y_0$. By Lemma \ref{prob2}, this in turn gives us an upper bound on the number of elements, in a typical orbit, that do not belong to $\mathbb{T} \times G$. Finally, in Lemma \ref{proof2} we obtain a uniform bound for the Lyapunov exponent, using Lemma \ref{counting lemma}. 
 
\begin{lemma}\label{proof1}
At most $\big[ \frac{b}{4(l+1)} \big]$ of the intervals in $\{I_j\}_{j=1}^b$ are bad. 
And for each interval $I_{j_1 \cdots j_{n}}$, at most $\big[ \frac{b}{4(l+1)} \big]$ of the intervals in $\{I_{j_1 \cdots j_{n} j}\}_{j=1}^{b}$ are bad.
\end{lemma}

\begin{proof}
Recall that, the set $I_j$ is bad if $I_j \cap \{x\in \mathbb{T}: f(x,y_0)\notin G\} \neq \emptyset$. From Definition \ref{definition f}(b) and since $\varepsilon < \varepsilon'$, the set $\{x\in \mathbb{T}: f(x,y_0)\notin G\}$ has at most $s$ intervals and has measure at most $\varepsilon'$. Since $|I_j|=1/b$, the maximum number of intervals in $\{I_j\}_{j=1}^{b}$ that can intersect a set of measure $\varepsilon'$, and consisting of $s$ intervals is given by 
\begin{align*}
    \#\left\{1\leq j \leq b:\mbox{$I_j$ is bad} \right\} \leq 2(s+1)+[\varepsilon'b].
\end{align*}

Let us fix the notation $q:= 2(s+1)+[\varepsilon'b]$. For any $n\geq 1$, let us consider the set $\left\{1\leq j \leq b: I_{j_1\cdots j_n j} \mbox{ is bad}\right\}$. By definition, the interval $I_{j_1\cdots j_n j}$ is bad if and only if there exists $x\in \mathbb{T}$ such that $\phi_{j_1 \cdots j_n j}(x) \notin G$. Since $T_j^{-1}:\mathbb{T} \rightarrow I_j$ is a bijective function, there exists $x\in \mathbb{T}$ for which $\phi_{j_1\cdots j_nj}(x) \notin G$ if and only if there exists $x\in I_j$ such that $f_x(\phi_{j_1\cdots j_n}(x)) \notin G$. Thus, for each $j_1, \dots, j_n,j\in \{1,\dots ,b\}$, the interval $I_{j_1 \cdots j_n j}$ is bad if and only if 
\begin{align}\label{proof1 eqn 0}
    \{x\in I_j: f_x(\phi_{j_1\cdots j_n}(x)) \notin G\} \neq \emptyset.
\end{align}
Using the upper bound on $|\phi'_{j_1\cdots j_n}(x)|$, from Lemma \ref{derivative}, we obtain 
\begin{align*}
    |\phi_{j_1\cdots j_n}(x_1)-\phi_{j_1\cdots j_n}(x_2)| \leq \left(\frac{R}{b}+\cdots +\frac{R^n}{b^n} \right)|x_1-x_2|,
\end{align*}
 for any $x_1, x_2 \in \mathbb{T} \setminus \{0\}$.
For $b$ big enough such that $\sum_{i=1}^{\infty}(\frac{R}{b})^i < \delta$, where $\delta$ is as defined in Lemma \ref{delta}, we have 
\begin{align*}
    |\phi_{j_1\cdots j_n}(x_1)-\phi_{j_1\cdots j_n}(x_2)| < \delta,
\end{align*}
for all $x_1, x_2 \in \mathbb{T}$. From the choice of $\delta$, for each $\phi_{j_1\cdots j_n}$ there exists $y_{j_1\cdots j_n} \in \mathbb{T}$ such that $\{\phi_{j_1\cdots j_n}(x): x\in \mathbb{T}\} \subset B_{\delta}(y_{j_1\cdots j_n})$, and thus 
\begin{align*}
    \{x\in \mathbb{T}: f_x(\phi_{j_1\cdots j_n}(x)) \notin G\} \subset \big\{ x\in \mathbb{T}: f_x(y) \notin G \text{ for some } y \in B_{\delta}(y_{j_1 \cdots j_n})\big\}.
\end{align*}
Then, we have
\begin{align}\label{proof1 eqn 1}
\bigcup_{j=1}^{b} \{x\in I_j: f_x(\phi_{j_1\cdots j_n}(x)) \notin G\} & \subset \big\{ x\in \mathbb{T}: f_x(y) \notin G \text{ for some } y \in B_{\delta}(y_{j_1, \cdots j_n})\big\}.
\end{align}
Thus, from equation (\ref{proof1 eqn 0}) and (\ref{proof1 eqn 1}), the set $I_{j_1\cdots j_n j}$ is bad only if 
\begin{align}\label{proof 1 eqn 2}
    I_j \cap \{x\in \mathbb{T}: f_x(y) \notin G \mbox{ for some } y\in B_{\delta}(y_{j_1\dots j_n})\} \neq \emptyset.
\end{align}

Each interval $I_j$ is of length $1/b$ and, from Lemma \ref{delta}, the set $\big\{x\in \mathbb{T}: f_x(y) \notin G \text{ for some } y \in B_{\delta}(y_{j_1\cdots j_n})\big\}$ can be covered by a set consisting of at most $s$ intervals and measure at most $\varepsilon'$. Thus, the number of intervals $I_j$ that satisfy equation (\ref{proof 1 eqn 2}) is bounded, and we have 
\begin{align*}
    \# \left\{1\leq j \leq b: I_{j_1\cdots j_nj} \mbox{ is bad } \right\} \leq 2(s+1)+[\varepsilon'b] = q.
\end{align*}
This proves that $q$ is the maximum number of bad sets in $\{I_{j_1\cdots j_n j}\}_{j=1}^{b}$. 

For any $f \in \mathcal{F}(\mathbb{T}^2)$, from Definition \ref{definition f}(b) and the chosen constant $l$ in Subsection \ref{counting lemma}, we have $\varepsilon < \tfrac{1}{6(l+1)}$. Thus, we can choose $\varepsilon'$ small enough such that $\varepsilon < \varepsilon'<\frac{1}{6(l+1)}$, and $b$ big enough, with respect to $s$ and $l$, such that 
\begin{align*}
    q= 2(s+1)+[\varepsilon'b] < \left[\frac{b}{4(l+1)}\right].
\end{align*}
This proves the lemma.
\end{proof}

\begin{lemma}\label{proof2}
For each fixed $y \in \mathbb{T}$ and for $a.e.$ $x \in \mathbb{T} $, we have $$L(x, y) < \frac{\log C}{2} < 0.$$
\end{lemma}

\begin{proof}

For each $(x,y)\in \mathbb{T}^2$ and integer $n \geq 1$, there is a corresponding sequence of intervals, that contain $x$, in the partitions in $\{\mathcal{P}_j\}_{j=1}^{n}$ that are defined in Subsection \ref{probability}. Let us denote it by $A_n(x) = \{I_{j_1}, I_{j_1 j_2},\dots, I_{j_1\cdots j_n}\}$. By Lemma \ref{prob2}, for a.e. $x \in \mathbb{T}$, there exists $N_0\in \mathbb{N}$ such that for all $N\geq N_0$, 
\begin{equation*}
    \# \{I \in A_N(x): I \text{ is bad }\} \leq [2(q/b)N].
\end{equation*}
By our choice of $b$, we know that 
$$ Np := 2(q/b)N < \frac{N}{2(l+1)}.$$
Then, by Lemma \ref{counting}, we have 
$$\prod_{j=0}^{N-1} |\partial_{y} f(x_j, y_j)| < C^{N/2},$$
for all $n\geq N_0$. From the expression of Lyapunov exponent in equation (\ref{lyapunov}), we obtain 
\begin{align*}
L(x,y)  \leq \limsup_{N\rightarrow \infty} \ \frac{1}{N} \log (C^{N/2}) < \frac {\log (C)}{2},
\end{align*}
and this proves Theorem \ref{theorem}. Moreover, this captures additional information about the full measure set of points $(x,y) \in \mathbb{T}^2$ that have a negative Lyapunov exponent.
\end{proof}



\section{Acknowledgements}
I would like to thank Kristian Bjerkl\"{o}v for his continued support and guidance. I am also grateful for the valuable insights and references provided by Reza Mohammadpour.


\subsection*{Disclosure Statement}
The research conveyed in this paper was funded by KTH Royal Institute of Technology, supported by the Swedish Research Council (grant 2020-03989), and there is no financial or non-financial competing interests to report.


\bibliographystyle{acm}
\bibliography{bibliography} 

\end{document}